\documentclass[11pt]{article}

\usepackage[margin=1.15in]{geometry}
\usepackage{amsmath,amssymb,amsthm,mathtools}
\usepackage{enumitem}
\usepackage{hyperref}

\newtheorem{theorem}{Theorem}
\newtheorem{lemma}[theorem]{Lemma}
\newtheorem{proposition}[theorem]{Proposition}

\newcommand{\R}{\mathbb R}
\newcommand{\Z}{\mathbb Z}
\newcommand{\E}{\mathbb E}
\newcommand{\Prob}{\mathbb P}
\newcommand{\cH}{\mathcal H}
\newcommand{\cL}{\mathcal L}
\newcommand{\eps}{\varepsilon}
\newcommand{\diam}{\operatorname{diam}}
\newcommand{\dist}{\operatorname{dist}}
\newcommand{\Var}{\operatorname{Var}}

\title{Randomly Shifted Steinhaus Longimeters and Buffon Discrepancy}
\author{Samuel Korsky}
\date{May 10, 2026}

\begin{document}

\maketitle

\begin{abstract}
\noindent
Let $\Omega\subset\R^2$ be a bounded convex domain. Steinerberger (2026) introduced the Buffon discrepancy problem: given length $L$, construct a one-dimensional set $S\subset\Omega$ such that the number of intersections of $S$ with a line $\ell$ approximates the Crofton-normalized chord length
\[
        \frac{2L}{\pi|\Omega|}\cdot\cH^1(\ell\cap\Omega).
\]
Steinerberger proved a universal upper bound of order $L^{1/3}$ using a Steinhaus longimeter construction, and showed that the disk admits bounded discrepancy. We prove that a randomly shifted Steinhaus construction improves the order of the universal upper bound to $L^{1/5}(\log L)^{2/5}$. 
\end{abstract}

\section{Introduction}

Let $\Omega\subset\R^2$ be a compact convex body with nonempty interior. Given a rectifiable one-dimensional set $S\subset\Omega$, one expects, by Crofton-type considerations, that the number of intersections of a typical line $\ell$ with $S$ should be proportional to the length of the chord $\ell\cap\Omega$. Steinerberger \cite{SteinerbergerBuffon} formalized this as the Buffon discrepancy problem.

\bigskip
\noindent
We write $|\Omega|$ for the area of $\Omega$ and $\cH^1$ for one-dimensional Hausdorff measure. For a finite union of line segments $S\subset\Omega$ with length $\cH^1(S) = L$, let $\#(\ell \cap S)$ denote the number of transverse intersections of $\ell$ with $S$, counted with multiplicity. The exceptional set of lines that contain a segment of $S$, pass through an endpoint, or pass through an intersection point of two segments has measure zero in the line space. Thus we define
\[
        \operatorname{disc}_\Omega(S)
        =
        \operatorname*{ess\,sup}_{\ell\in\cL}
        \left|
        \#(\ell \cap S)
        -
        \frac{2L}{\pi|\Omega|}
        \cdot \cH^1(\ell\cap\Omega)
        \right|.
\]
Here $\cL$ denotes the space of unoriented lines in $\R^2$, equipped with the usual motion-invariant measure.

\bigskip
\noindent
Steinerberger proved that for every convex $\Omega$ there are sets $S$ of length $L$ satisfying
\[
        \operatorname{disc}_\Omega(S)\lesssim_\Omega L^{1/3}.
\]
His construction generalizes the Steinhaus longimeter: choose $n$ evenly spaced directions, draw parallel lines with spacing $\eps$ in each direction, and restrict the resulting grid to $\Omega$. If $n/\eps\sim L$, the discrepancy is (up to a constant) bounded by
\[
        \frac{L}{n^2}+n.
\]
The first term is an angular quadrature error. The second term is a crude endpoint-rounding error, obtained by summing an $O(1)$ error over $n$ directions. Optimizing gives $n\sim L^{1/3}$ and discrepancy $O_\Omega(L^{1/3})$.

\bigskip
\noindent
The purpose of this note is to show that the endpoint errors do not need to add coherently. We randomly shift each of the $n$ parallel-line families so that for any fixed chord, the endpoint errors are independent, centered, and bounded. A simple deterministic net argument then gives a uniform estimate of order
\[
        \sqrt{n\log L}+\log L.
\]
This improves the error model to
\[
        \frac{L}{n^2}+\sqrt{n\log L}+\log L.
\]
Choosing
\[
        n\sim L^{2/5}(\log L)^{-1/5}
\]
then gives the following theorem:

\begin{theorem}\label{thm:main}
Let $\Omega\subset\R^2$ be a compact convex body with nonempty interior. Then there exists $C_\Omega<\infty$ such that for every sufficiently large $L$ there is a finite union of line segments $S\subset\Omega$ satisfying
\[
        \cH^1(S)=L
\]
and
\[
        \operatorname{disc}_\Omega(S)
        \le
        C_\Omega L^{1/5}(\log L)^{2/5}.
\]
\end{theorem}

\section{The Randomly Shifted Construction}

Let
\[
        A=|\Omega|,
        \quad
        D=\diam(\Omega).
\]
For an integer $n\ge1$, define
\[
        \nu_k
        =
        \left(
        \cos\frac{\pi k}{n},
        \sin\frac{\pi k}{n}
        \right),
        \qquad
        0\le k<n.
\]
Let $U_0,\dots,U_{n-1}$ be independent uniform random variables on $[0,1)$. For a spacing $\eps>0$, define the shifted line family
\[
        \mathcal L_k
        =
        \left\{
        x\in\R^2:
        x\cdot\nu_k=\eps(q+U_k)
        \text{ for some }q\in\Z
        \right\}.
\]
The randomly shifted Steinhaus set is
\[
        S_{n,\eps,U}
        =
        \Omega\cap\bigcup_{k=0}^{n-1}\mathcal L_k.
\]
Since overlaps of different line families are finite sets, length is additive:
\[
        \cH^1(S_{n,\eps,U})
        =
        \sum_{k=0}^{n-1}
        \cH^1(\Omega\cap\mathcal L_k).
\]
Let $\ell$ be a line such that $\ell\cap\Omega$ is a nondegenerate segment. Write
\[
        \ell\cap\Omega=[x,y],
        \quad
        h_\ell=|x-y|.
\]
Let
\[
        t=\frac{y-x}{|y-x|}
\]
be the chord direction. For $0\le k<n$, define
\[
        a_k(x,y)=\min(x\cdot\nu_k,y\cdot\nu_k),
        \quad
        b_k(x,y)=\max(x\cdot\nu_k,y\cdot\nu_k).
\]
The number of transverse intersections of $\ell$ with the $k$-th shifted family equals
\[
        N_k(x,y)
        =
        \#\left\{
        q\in\Z:
        \eps(q+U_k)\in [a_k(x,y),b_k(x,y))
        \right\},
\]
where the half-open interval convention removes endpoint ambiguity. The projected length of the chord in direction $\nu_k$ is
\[
        b_k(x,y)-a_k(x,y)
        =
        |(y-x)\cdot\nu_k|
        =
        h_\ell |t\cdot\nu_k|.
\]
Thus
\[
        N_k(x,y)
        =
        \frac{h_\ell |t\cdot\nu_k|}{\eps}
        +
        \xi_k(x,y),
\]
where
\[
        \xi_k(x,y)
        =
        N_k(x,y)
        -
        \frac{b_k(x,y)-a_k(x,y)}{\eps}.
\]
We will estimate separately the deterministic main term
\[
        \frac{h_\ell}{\eps}\sum_{k=0}^{n-1}|t\cdot\nu_k|
\]
and the random endpoint error
\[
        Z(x,y)=\sum_{k=0}^{n-1}\xi_k(x,y).
\]

\section{The Angular Quadrature Estimate}
Here we proceed exactly as in \cite{SteinerbergerBuffon}, with the following lemma:

\begin{lemma}[Angular quadrature]\label{lem:angular}
Uniformly for every unit vector $t\in S^1$,
\[
        \sum_{k=0}^{n-1}|t\cdot\nu_k|
        =
        \frac{2n}{\pi}
        +
        O\left(\frac1n\right).
\]
\end{lemma}

\begin{proof}
Write
\[
        t=(\cos\theta,\sin\theta).
\]
Then
\[
        |t\cdot\nu_k|
        =
        \left|
        \cos\left(\theta-\frac{\pi k}{n}\right)
        \right|.
\]
The Fourier series for $|\cos u|$ is
\[
        |\cos u|
        =
        \frac2\pi
        +
        \frac4\pi
        \sum_{r=1}^\infty
        \frac{(-1)^r}{1-4r^2}\cos(2ru).
\]
The coefficients are $O(r^{-2})$. Summing over $k$ gives
\[
        \sum_{k=0}^{n-1}
        \cos\left(2r\left(\theta-\frac{\pi k}{n}\right)\right)
        =
        \begin{cases}
        n\cos(2r\theta), & n\mid r,\\
        0, & n\nmid r.
        \end{cases}
\]
Therefore
\[
\begin{aligned}
        \sum_{k=0}^{n-1}|t\cdot\nu_k|
        &=
        \frac{2n}{\pi}
        +
        n\sum_{j=1}^{\infty}O\left(\frac1{(jn)^2}\right)  \\
        &=
        \frac{2n}{\pi}+O\left(\frac1n\right),
\end{aligned}
\]
uniformly in $\theta$.
\end{proof}

\bigskip
\noindent
Consequently,
\[
        \frac{h_\ell}{\eps}
        \sum_{k=0}^{n-1}|t\cdot\nu_k|
        =
        \frac{2n}{\pi\eps}\cdot h_\ell
        +
        O_\Omega\left(\frac1{\eps n}\right).
\]

\section{The Endpoint Error}

The main probabilistic estimate, and the main contribution of this note, is the following:

\begin{lemma}[Uniform endpoint error]\label{lem:endpoint}
Let $\Omega\subset\R^2$ be compact and contained in a ball of radius $D$. Assume $0<\eps\le1$ and $n\ge1$. There exist constants
\[
        C_R=C_R(\Omega)\ge e,
        \quad
        C_E=C_E(\Omega)<\infty
\]
such that, if
\[
        R=\frac{C_R n}{\eps},
\]
then, with probability at least $1-2R^{-10}$,
\[
        \sup_{x,y\in\Omega}
        \left|
        \sum_{k=0}^{n-1}
        \left(
        N_k(x,y)
        -
        \frac{b_k(x,y)-a_k(x,y)}{\eps}
        \right)
        \right|
        \le
        C_E
        \left(
        \sqrt{n\log R}+\log R
        \right).
\]
\end{lemma}

\begin{proof}
We write
\[
        Z(x,y)=
        \sum_{k=0}^{n-1}
        \left(
        N_k(x,y)
        -
        \frac{b_k(x,y)-a_k(x,y)}{\eps}
        \right).
\]
Let
\[
        T=K_E\left(\sqrt{n\log R}+\log R\right),
\]
where $K_E=K_E(\Omega)$ will be chosen large enough below. At the end of the proof we set $C_E=2K_E$.

\bigskip
\noindent
If $T\ge n$, then the result is trivial, since each summand is bounded in absolute value by $1$, so $|Z(x,y)|\le n\le T$. Hence assume $T<n$. In this case
\[
        \delta:=\frac{\eps T}{100n}<\frac{\eps}{100}.
\]
We complete the proof in four steps.

\subsection{Fixed Pairs}

Fix $x,y\in\Omega$ and $k$. Put
\[
        a=a_k(x,y),
        \quad
        b=b_k(x,y),
        \quad
        d=b-a.
\]
Then
\[
        N_k(x,y)
        =
        \#\{q\in\Z:\eps(q+U_k)\in[a,b)\}.
\]
For every shift $U_k$, the number of shifted lattice points in an interval of length $d$ is either
\[
        \left\lfloor\frac d\eps\right\rfloor
        \quad\text{or}\quad
        \left\lfloor\frac d\eps\right\rfloor+1.
\]
Hence
\[
        \left|
        N_k(x,y)-\frac d\eps
        \right|\le1.
\]
Moreover,
\[
\begin{aligned}
        \E_{U_k}\left[N_k(x,y)\right]
        &=
        \int_0^1
        \sum_{q\in\Z}
        \mathbf 1_{\eps(q+u)\in[a,b)}
        \,du  \\
        &=
        \frac1\eps\int_a^b ds
        =
        \frac d\eps.
\end{aligned}
\]
Therefore the summands in $Z(x,y)$ are independent, centered, and bounded by $1$. Hoeffding's inequality \cite{Hoeffding1963} gives
\[
        \Prob\{|Z(x,y)|>s\}
        \le
        2\exp\left(-\frac{2s^2}{n}\right).
\]

\subsection{Deterministic Endpoint Net}

Choose a deterministic $\delta$-net $\mathcal P\subset\Omega$. Since $\Omega$ is bounded, there is a constant
\[
        C_{\mathrm{net}}=C_{\mathrm{net}}(\Omega)<\infty
\]
such that, for every $0<\delta\le1$, one can choose $\mathcal P$ with
\[
        |\mathcal P|\le C_{\mathrm{net}}\delta^{-2}.
\]
Since
\[
        \delta=\frac{\eps T}{100n},
\]
we have
\[
        |\mathcal P|
        \le
        C_{\mathrm{net}}
        \left(\frac{100n}{\eps T}\right)^2
        =
        10^4 C_{\mathrm{net}}
        \left(\frac{n}{\eps T}\right)^2.
\]
Because $T\ge K_E\log R\ge1$, this implies
\[
        |\mathcal P|
        \le
        10^4 C_{\mathrm{net}}
        \left(\frac{n}{\eps}\right)^2.
\]
We now choose the structural constant $C_R=C_R(\Omega)$ in the definition of $R$ large enough that
\[
        C_R^2\ge 10^4 C_{\mathrm{net}}.
\]
Then
\[
        |\mathcal P|\le R^2.
\]

\bigskip
\noindent
For each pair $(p,q)\in\mathcal P^2$, Hoeffding's inequality gives
\[
        \Prob\{|Z(p,q)|>T\}
        \le
        2\exp\left(-\frac{2T^2}{n}\right).
\]
Since $|\mathcal P|^2\le R^4$, a union bound gives
\[
        \Prob\left\{
        \exists p,q\in\mathcal P:\ |Z(p,q)|>T
        \right\}
        \le
        2R^4\exp\left(-\frac{2T^2}{n}\right).
\]
Since $T^2/n\ge K_E^2\log R$, choosing $K_E$ sufficiently large gives
\[
        2R^4\exp\left(-\frac{2T^2}{n}\right)
        \le
        R^{-10}.
\]
Thus, with probability at least $1-R^{-10}$,
\[
        |Z(p,q)|\le T
        \quad
        \text{for all }p,q\in\mathcal P.
\]

\subsection{Unstable Net Points}

For $p\in\mathcal P$ and $0\le k<n$, call $p$ unstable in direction $k$ if
\[
        \dist_{\R/\Z}\left(
        \frac{p\cdot\nu_k}{\eps}-U_k,\ 
        \Z
        \right)
        \le
        \frac{\delta}{\eps}.
\]
Equivalently, $p$ lies within projected distance $\delta$ of one of the shifted lines in the $k$-th family. Since $\delta/\eps<1/100$, for fixed $p,k$,
\[
        \Prob(p\text{ is unstable in direction }k)
        \le
        \frac{2\delta}{\eps}
        =
        \frac{T}{50n}.
\]
Let
\[
        B(p)=
        \#\{0\le k<n:\ p\text{ is unstable in direction }k\}.
\]
Then $B(p)$ is stochastically dominated by a binomial random variable with mean at most $T/50$. By the Chernoff bound \cite{Chernoff1952},
\[
        \Prob\{B(p)>T/5\}
        \le
        \exp(-c_{\mathrm{ch}}T),
\]
where $c_{\mathrm{ch}}>0$ is an absolute constant. Since $|\mathcal P|\le R^2$ and $T\ge K_E\log R$, another union bound gives
\[
        \Prob\left\{
        \exists p\in\mathcal P:\ B(p)>T/5
        \right\}
        \le
        R^2\exp(-c_{\mathrm{ch}}T)
        \le
        R^{-10},
\]
provided $K_E$ is large enough.

\bigskip
\noindent
Thus, with probability at least $1-R^{-10}$,
\[
        B(p)\le T/5
        \quad
        \text{for all }p\in\mathcal P.
\]

\subsection{Extending the Net to All Endpoints}

Assume the two high-probability events from the previous two subsections both occur.

\bigskip
\noindent
Let $x,y\in\Omega$ be arbitrary. Choose $p,q\in\mathcal P$ such that
\[
        |x-p|\le\delta,
        \quad
        |y-q|\le\delta.
\]
We compare $Z(x,y)$ and $Z(p,q)$.

\bigskip
\noindent
For a fixed direction $k$, the count $N_k(x,y)$ can differ from $N_k(p,q)$ only if, while moving $p$ to $x$ or $q$ to $y$, one of the endpoints crosses a shifted lattice line of the $k$-th family. This can occur only if $p$ or $q$ is unstable in direction $k$. Hence
\[
        |N_k(x,y)-N_k(p,q)|
        \le
        \mathbf 1_{\{p\text{ unstable in }k\}}
        +
        \mathbf 1_{\{q\text{ unstable in }k\}}.
\]
Summing in $k$,
\[
        \sum_{k=0}^{n-1}
        |N_k(x,y)-N_k(p,q)|
        \le
        B(p)+B(q)
        \le
        \frac{2T}{5}.
\]
The deterministic mean term changes by at most
\[
\begin{aligned}
        &
        \sum_{k=0}^{n-1}
        \frac{
        \left|
        |(x-y)\cdot\nu_k|
        -
        |(p-q)\cdot\nu_k|
        \right|
        }{\eps}  \\
        &\qquad\le
        \sum_{k=0}^{n-1}
        \frac{
        |(x-p)\cdot\nu_k|+|(y-q)\cdot\nu_k|
        }{\eps}  \\
        &\qquad\le
        \frac{2n\delta}{\eps}
        =
        \frac{T}{50}.
\end{aligned}
\]
Therefore
\[
        |Z(x,y)-Z(p,q)|
        \le
        \frac{2T}{5}+\frac{T}{50}
        \le
        T.
\]
Since $|Z(p,q)|\le T$, we obtain
\[
        |Z(x,y)|\le 2T.
\]
Taking the supremum over $x,y\in\Omega$ and recalling that $C_E=2K_E$ proves the lemma.
\end{proof}

\section{Length Concentration}

We also need to know that the random construction has length close to its expectation.

\begin{lemma}[Length concentration]\label{lem:length}
Let
\[
        L_U=\cH^1(S_{n,\eps,U}).
\]
There exists $C_L=C_L(\Omega)<\infty$ such that
\[
        \E \left[L_U\right]=\frac{n|\Omega|}{\eps},
\]
and, with $R=C_R n/\eps$ as in Lemma \ref{lem:endpoint},
\[
        \Prob\left\{
        \left|
        L_U-\frac{n|\Omega|}{\eps}
        \right|
        >
        C_L\sqrt{n\log R}
        \right\}
        \le
        R^{-10}.
\]
\end{lemma}

\begin{proof}
Fix $k$. Define the slice-length function
\[
        g_k(s)
        =
        \cH^1\left(
        \Omega\cap\{x:x\cdot\nu_k=s\}
        \right).
\]
Then
\[
        \cH^1(\Omega\cap\mathcal L_k)
        =
        \sum_{q\in\Z}g_k(\eps(q+U_k)).
\]
Averaging over $U_k$ gives
\[
\begin{aligned}
        \E_{U_k}\left[\cH^1(\Omega\cap\mathcal L_k)\right]
        &=
        \int_0^1
        \sum_{q\in\Z}
        g_k(\eps(q+u))
        \,du  \\
        &=
        \frac1\eps\int_\R g_k(s)\,ds
        =
        \frac{|\Omega|}{\eps}.
\end{aligned}
\]
Summing in $k$ yields
\[
        \E \left[L_U\right]=\frac{n|\Omega|}{\eps}.
\]
Since $\Omega$ is convex, $g_k$ is concave on its support. Hence it has bounded variation, with
\[
        \Var(g_k)\le 2\sup_s g_k(s)\le 2D.
\]
For any compactly supported function $g$ of bounded variation,
\[
        \left|
        \sum_{q\in\Z}g(\eps(q+u))
        -
        \frac1\eps\int_\R g(s)\,ds
        \right|
        \le
        \Var(g).
\]
Indeed, after multiplying by $\eps$, compare each sampled value with the average of $g$ on the corresponding interval of length $\eps$:
\[
\begin{aligned}
        &
        \left|
        \eps\sum_{q\in\Z}g(\eps(q+u))
        -
        \int_\R g(s)\,ds
        \right|  \\
        &\qquad\le
        \sum_{q\in\Z}
        \int_{\eps(q+u)}^{\eps(q+u+1)}
        |g(\eps(q+u))-g(s)|\,ds
        \le
        \eps\,\Var(g).
\end{aligned}
\]
Dividing by $\eps$ gives the claim.

\bigskip
\noindent
Thus each directional length contribution satisfies
\[
        \left|
        \cH^1(\Omega\cap\mathcal L_k)
        -
        \frac{|\Omega|}{\eps}
        \right|
        \le
        2D.
\]
The $n$ contributions are independent. Hoeffding's inequality gives
\[
        \Prob\left\{
        \left|
        L_U-\frac{n|\Omega|}{\eps}
        \right|>s
        \right\}
        \le
        2\exp\left(-\frac{2s^2}{n(4D)^2}\right).
\]
Taking
\[
        s=C_L\sqrt{n\log R}
\]
with $C_L=C_L(\Omega)$ sufficiently large proves the lemma.
\end{proof}

\section{The Discrepancy Bound}

\begin{proposition}\label{prop:random}
Let
\[
        M=\frac{n|\Omega|}{\eps}
\]
be the expected length of $S_{n,\eps,U}$, and let
\[
        R=\frac{C_R n}{\eps}
\]
with $C_R$ as in Lemma \ref{lem:endpoint}. There exists $C_P=C_P(\Omega)<\infty$ such that, with positive probability,
\[
        \operatorname{disc}_\Omega(S_{n,\eps,U})
        \le
        C_P
        \left(
        \frac{M}{n^2}
        +
        \sqrt{n\log R}
        +
        \log R
        \right),
\]
and
\[
        \left|
        \cH^1(S_{n,\eps,U})-M
        \right|
        \le
        C_P\sqrt{n\log R}.
\]
\end{proposition}

\begin{proof}
Let $\ell$ be a generic line, so that $\ell\cap\Omega=[x,y]$ is a segment and all intersections with the constructed grid are transverse. Let
\[
        h_\ell=\cH^1(\ell\cap\Omega),
        \quad
        t=\frac{y-x}{|y-x|}.
\]
Then
\[
\begin{aligned}
        \#\left(\ell \cap {S_{n,\eps,U}}\right)
        &=
        \sum_{k=0}^{n-1}N_k(x,y)  \\
        &=
        \frac{h_\ell}{\eps}
        \sum_{k=0}^{n-1}|t\cdot\nu_k|
        +
        Z(x,y).
\end{aligned}
\]
By Lemma \ref{lem:angular},
\[
        \frac{h_\ell}{\eps}
        \sum_{k=0}^{n-1}|t\cdot\nu_k|
        =
        \frac{2n}{\pi\eps}\cdot h_\ell
        +
        O_\Omega\left(\frac1{\eps n}\right).
\]
Since
\[
        M=\frac{n|\Omega|}{\eps},
\]
we have
\[
        \frac{2n}{\pi\eps}
        =
        \frac{2M}{\pi|\Omega|}.
\]
Also,
\[
        \frac1{\eps n}
        =
        \frac{M}{|\Omega|n^2}
        \lesssim_\Omega \frac{M}{n^2}.
\]
Therefore
\[
          \#\left(\ell \cap {S_{n,\eps,U}}\right)
        =
        \frac{2M}{\pi|\Omega|}\cdot h_\ell
        +
        O_\Omega\left(\frac{M}{n^2}\right)
        +
        Z(x,y).
\]
By Lemma \ref{lem:endpoint}, with probability at least $1-2R^{-10}$,
\[
        \sup_{x,y\in\Omega}|Z(x,y)|
        \le
        C_E
        \left(
        \sqrt{n\log R}+\log R
        \right).
\]
By Lemma \ref{lem:length}, with probability at least $1-R^{-10}$,
\[
        \left|
        \cH^1(S_{n,\eps,U})-M
        \right|
        \le
        C_L\sqrt{n\log R}.
\]
For $R$ sufficiently large, these two events have positive simultaneous probability. On this event, replacing $M$ by the actual length
\[
        L_U=\cH^1(S_{n,\eps,U})
\]
in the Crofton normalization changes the target term by at most
\[
        \frac{2|L_U-M|}{\pi|\Omega|}\cdot
        \cH^1(\ell\cap\Omega)
        \le
        \frac{2DC_L}{\pi|\Omega|}\sqrt{n\log R},
\]
because $\cH^1(\ell\cap\Omega)\le D$. Taking the essential supremum over $\ell$ and absorbing the displayed constants into $C_P$ gives
\[
        \operatorname{disc}_\Omega(S_{n,\eps,U})
        \le
        C_P
        \left(
        \frac{M}{n^2}
        +
        \sqrt{n\log R}
        +
        \log R
        \right).
\]
The length estimate also follows after increasing $C_P$ so that $C_P\ge C_L$.
\end{proof}

\section{Optimization and Exact Length}

We now prove Theorem \ref{thm:main}. Let
\[
        \Phi(L)=L^{1/5}(\log L)^{2/5}.
\]
It suffices to consider large $L$.

\bigskip
\noindent
Choose an expected length
\[
        M=L-K_0 \Phi(L),
\]
where $K_0=K_0(\Omega)$ is a constant to be fixed. For large $L$, we have $M\sim L$ and
\[
        \Phi(M)\sim \Phi(L).
\]
Choose
\[
        n=
        \left\lfloor
        M^{2/5}(\log M)^{-1/5}
        \right\rfloor
\]
and set
\[
        \eps=\frac{n|\Omega|}{M}.
\]
Then
\[
        \frac{n}{\eps}=\frac{M}{|\Omega|},
\]
so, with $C_R$ as in Lemma \ref{lem:endpoint},
\[
        R=\frac{C_R n}{\eps}
        =
        \frac{C_R}{|\Omega|}M
        \asymp_\Omega M.
\]
By Proposition \ref{prop:random}, there is a deterministic choice of phases $U$ such that
\[
        \operatorname{disc}_\Omega(S_{n,\eps,U})
        \le
        C_P
        \left(
        \frac{M}{n^2}
        +
        \sqrt{n\log R}
        +
        \log R
        \right),
\]
and
\[
        \left|
        \cH^1(S_{n,\eps,U})-M
        \right|
        \le
        C_P\sqrt{n\log R}.
\]
Since $R\asymp_\Omega M$, the chosen value of $n$ gives
\[
        \frac{M}{n^2}
        \lesssim_\Omega
        M^{1/5}(\log M)^{2/5},
\]
and
\[
        \sqrt{n\log R}
        \lesssim_\Omega
        M^{1/5}(\log M)^{2/5}.
\]
Also,
\[
        \log R
        =
        o\left(M^{1/5}(\log M)^{2/5}\right).
\]
Consequently, there is a constant $C_{\mathrm{opt}}=C_{\mathrm{opt}}(\Omega)$ such that
\[
        \operatorname{disc}_\Omega(S_{n,\eps,U})
        \le
        C_{\mathrm{opt}}\Phi(L),
\]
and
\[
        \left|
        \cH^1(S_{n,\eps,U})-M
        \right|
        \le
        C_{\mathrm{opt}}\Phi(L).
\]
Choose
\[
        K_0\ge 2C_{\mathrm{opt}}.
\]
Then
\[
        \cH^1(S_{n,\eps,U})\le L
\]
and
\[
        0\le
        L-\cH^1(S_{n,\eps,U})
        \le
        (K_0+C_{\mathrm{opt}})\Phi(L)
        \lesssim_\Omega \Phi(L).
\]
It remains to adjust the length upward to exactly $L$. Let
\[
        \Delta=L-\cH^1(S_{n,\eps,U}).
\]
Choose a closed disk $B\subset\Omega$ of radius $\rho_\Omega>0$. Inside $B$, add a finite collection $E$ of pairwise disjoint line segments with total length exactly $\Delta$, arranged so that no segment of $E$ overlaps any segment of $S_{n,\eps,U}$. This can be done with at most
\[
        C_{\mathrm{add}}(\Omega)(1+\Delta)
\]
segments. For example, choose a direction not parallel to any of the finitely many directions already used, place disjoint short parallel segments inside $B$, and shorten the final segment so the total length is exactly $\Delta$.

\bigskip
\noindent
Now, let
\[
        S=S_{n,\eps,U}\cup E.
\]
Then
\[
        \cH^1(S)=L.
\]
For every generic line $\ell$, the number of new intersections contributed by $E$ is at most the number of added segments, hence at most
\[
        C_{\mathrm{add}}(\Omega)(1+\Delta)
        \lesssim_\Omega
        \Phi(L).
\]
The change in the Crofton target term caused by increasing the length from $\cH^1(S_{n,\eps,U})$ to $L$ is at most
\[
        \frac{2\Delta}{\pi|\Omega|}
        \cH^1(\ell\cap\Omega)
        \le
        \frac{2D}{\pi|\Omega|}\Delta
        \lesssim_\Omega
        \Phi(L).
\]
Therefore
\[
        \operatorname{disc}_\Omega(S)
        \le
        C_\Omega\Phi(L)
        =
        C_\Omega L^{1/5}(\log L)^{2/5},
\]
for a final constant $C_\Omega$ depending only on $\Omega$. This proves Theorem \ref{thm:main}.

\section*{Acknowledgements} The author thanks Stefan Steinerberger for introducing the Buffon discrepancy problem and for suggesting, in \cite{SteinerbergerBuffon}, that offsets may avoid the large intersection phenomenon in the deterministic Steinhaus construction. The author also acknowledges GPT-5.5 for assistance in performing the detailed computations and preparing an initial draft of this preprint. The proof idea and the direction of the argument are due to the author.

\end{document}